\documentclass[12pt]{amsart}
\usepackage[T2A]{fontenc}
\usepackage[cp1251]{inputenc}
\usepackage[english]{babel}
\usepackage{amsmath,latexsym,amsthm,amsfonts,tipa,upgreek}
\usepackage{amssymb,amscd,graphpap, stmaryrd}

\newcommand\NN{{\mathbb N}}
\newcommand\ZZ{{\mathbb Z}}

\newcommand\w{{\omega}}

\newcommand\e{{\varepsilon}}

\newcommand\PP{{\mathcal P}}
\newcommand\FF{{\mathcal F}}

\newcommand\vt{{\Updelta}}
\newcommand\vm{{\bigtriangleup}}
\newcommand\J{J}

\newtheorem{theorem}{Theorem}[section]
\newtheorem{Th}[theorem]{Theorem}

\newtheorem{Lm}[theorem]{Lemma}
\newtheorem{Ps}[theorem]{Proposition}

\newtheorem{Qs}[theorem]{Question}
\newtheorem{Cr}[theorem]{Corollary}
\theoremstyle{definition}

\newtheorem{Rm}[theorem]{Remark}

\begin{document}

\title{On the subset Combinatorics of $G$-spaces}
\author{Igor Protasov and Sergii Slobodianiuk}


\date{}
\maketitle

\begin{abstract} Let $G$ be a group and let $X$ be a transitive $G$-space. 
We classify the subsets of $X$ with respect to a translation invariant ideal $\J$ in the Boolean algebra of all subsets of $X$, 
introduce and apply the relative combinatorical derivations of subsets of $X$. 
Using the standard action of $G$ on the Stone-$\check{C}$ech compactification $\beta X$ of the discrete space $X$, 
we characterize the points $p\in\beta X$ isolated in $Gp$ and describe a size of a subset of $X$ in terms of its ultracompanions in $\beta X$. 
We introduce and characterize scattered and sparse subsets of $X$ from different points of view. \end{abstract}

\section{Introduction}

Let $G$ be a group and let $X$ be a transitive $G$-space with the action $G\times X\to X$, $(g,x)\mapsto gx$. If $X=G$ and $gx$ is a product of $g$ and $x$ then $X$ is called the {\em left regular $G$-space}.

A family $\J$ of subsets of $X$ is called an ideal in the Boolean algebra $\PP_X$ of all subsets of $X$ if $X\notin\J$ and $A,B\in\J$, $C\subset A$ imply $A\cup B\in\J$ and $C\in\J$. The ideal of all finite subsets of $X$ is denoted by $[X]^{<\w}$. An ideal $\J$ is {\em translation invariant} if $gA\in\J$ for all $g\in G$, $A\in\J$, where $gA=\{ga:a\in A\}$. If $X$ is finite then $\J=\{\varnothing\}$ so in what follows all $G$-spaces are supposed to be infinite.

Now we fix a translation invariant ideal $\J$ in $\PP_X$ and say that a subset $A$ of $X$ is
\begin{itemize}
\item{} {\em $\J$-large} if $X=FA\cup I$ for some $F\in [G]^{<\w}$ and $I\in\J$;
\item{} {\em $\J$-small} if $L\setminus A$ is $\J$-large for every $J$-large subset $L$ of $X$;
\item{} {\em $\J$-thick} if $Int_F(A)\notin\J$ for each $F\in[G]^{<\w}$, where $Int_F(A)=\{a\in A:Fa\subseteq A\}$;
\item{} {\em $\J$-prethick} if $FA$ is thick for some $F\in[G]^{<\w}$.
\end{itemize}

If $\J=\varnothing$ we omit the prefix $\J$ and get a well-known classification of subsets of a $G$-spaces by their combinatorial size (see the survey \cite{b11}).

In the case of the left regular $G$-spaces, the notions of $J$-large and $\J$-small subsets appeared in \cite{b1}.

We say that a mapping $\vt_\J:\PP_X\to\PP_G$ defined by $$\vt_\J(A)=\{g\in G:gA\cap A\notin\J\}$$ is a {\em combinatorial derivation relatively to the ideal $\J$}. 
If  $X$ is the left regular $G$-space and $\J=[X]^{<\infty}$, the mapping $\vt_\J$ was introduced in \cite{b12} under the name combinatorial derivation and studied in \cite{b13}.

In Section $2$ we prove that if a subset $A$ of $X$ is not $\J$-small then $\vt_\J(A)$ is large in $G$. 
For the left regular $G$-space $X$ and $\J=[X]^{<\w}$, this statement was proved in \cite{b6}.

We endow $X$ with the discrete topology and take the points of $\beta X$, the Stone-$\check{C}$ech compactification of $X$, to be the ultrafilters on $X$, with the points of $X$ identified with the principal ultrafilters on $X$. 
The topology on $\beta X$ can be defined by stating that the set of the form $\overline{A}=\{p\in\beta X: A\in p\}$, where $A$ is a subset of $X$, form a base for the open sets. We note the sets of this form are clopen and that for any $p\in\beta X$ and $A\subset X$, $A\in p$ if and only if $p\in\overline{A}$. We denote $A^*=\overline{A}\cap X^*$, where $X^*=\beta X\setminus X$. The universal property of $\beta X$ states that every mapping $f: X\to Y$, where $Y$ is a compact Hausdorff space, can be extended to the continuous mapping $f^\beta:\beta X\to Y$.

Now we endow $G$ with the discrete topology and, using the universal property of $\beta G$, extend the group multiplication from $G$ to $\beta G$ (see \cite[Chapter 4]{b8}), so $\beta G$ becomes a compact right topological semigroup. 

We define the action of $\beta G$ on $\beta X$ in two steps. Given $g\in G$, the mapping $$x\mapsto gx: \text{  }X\to\beta X$$ extends to the continuous mapping $$p\mapsto gp:\text{ }\beta X\to \beta X.$$ Then, for each $p\in\beta X$, we extend the mapping $g\mapsto gp:\text{ }G\to\beta X$ to the continuous mapping $$q\mapsto qp:\text{ }\beta G\to\beta X.$$ Let $q\in\beta G$ and $p\in\beta X$. To describe a base for the ultrafilter $qp\in\beta X$, we take any element $Q\in q$ and, for every $g\in Q$ choose some element $P_x\in p$. Then $\bigcup_{g\in Q}gP_x\in qp$, and the family of subsets of this form is a base for the ultrafilter $qp$.

Given a subset $A$ of $X$ and an ultrafilter $p\in X^*$ we define a {\em $p$-companion} of $A$ by
$$\vm_p(A)=A^*\cap Gp=\{gp: g\in G, A\in gp\},$$
and say that a subset $S$ of $X^*$ is an {\em ultracompanion} of $A$ if $S=\vm_p(A)$ for some $p\in X^*$.

In Section $3$ we characterze the subsets of $X$ of different types in terms of their ultracompanions. 
For example a subset $A$ of $X$ is $\J$-large if and only if $\vm_p(A)\neq\varnothing$ for each $p\in\check{\J}$, where $\check{\J}=\{p\in X^*: X\setminus I\in p\text{ for every } I\in\J\}$. 
For the left regular $X$ and $\J=\{\varnothing\}$, these characterizations are obtained in \cite{b15}.

In Section $4$ we describe the points $p\in\beta X$ isolated in $Gp$ and introduce the piecewise shifted $FP$-sets in $X$ to characterize the subsets $A\subseteq X$ such that $\vm_p(A)$ is discrete for each $p\in X^*$.

In Section $5$ we extend the notions scattered and sparse subsets from groups \cite{b3} to $G$-space and characterize these subsets from different points of view.

\section{Relative combinatorial derivations}~\label{s2}

Let $X$ be a transitive $G$-space and let $\J$ be a translation invariant ideal in $\PP_X$.

\begin{Lm}\label{l2.1} For a subset $A$ of $X$, the following statements are equivalent
\begin{itemize}
\item[{\it (i)}] $A$ is $\J$-small;
\item[{\it (ii)}] $G\setminus FA$ is $\J$-large for each $F\in[G]^{<\w}$;
\item[{\it (iii)}] $A$ is not $\J$-prethick.
\end{itemize}
\end{Lm}
\begin{proof} Apply the arguments proving Theorem $2.1$ in \cite{b1}. \end{proof}

The next lemma follows directly from the definition of $\J$-small subsets.

\begin{Lm}\label{l2.2} The family of all $\J$-small subsets of $X$ is a translation invariant ideal in $\PP_X$. \end{Lm}

\begin{Lm}\label{l2.3} Let $L$ be a $\J$-large subset of $X$. 
Then given a partition $L=A\cup B$, either $\vt_\J(A)$ is large or $B$ is $\J$-large. \end{Lm}

\begin{proof} We take $F\in[G]^{<\w}$ and $I\in\J$ such that $G=F(A\cup B)\cup I$. 
Assume that $G\neq F\vt_\J(A)$ and show that $B$ is $\J$-large.

Let $F=\{f_1,...,f_k\}$. We take $g\in G\setminus F\vt_\J(A)$ and put $I_i=f_i^{-1}gA\cap A$, $i\in\{1,...,k\}$. 
Since $g\notin f_i\vt_\J(A)$, we have $I_i\in\J$ and $f_i^{-1}gx\notin A$ for each $x\in A\setminus I_i$.

If $x\in X$ and $F^{-1}gx\cap L=\varnothing$ then $gx\notin FL$ so $gx\in I$ and $x\in g^{-1}I$. We put 
$$I'=I_1\cup...\cup I_k\cup g^{-1}I.$$
If $x\in A\setminus I'$ then there is $i\in\{1,...,k\}$ such that $f_i^{-1}gx\in A\cup B$. Since $f_i^{-1}gx\notin A$, we have $f_i^{-1}gx\in B$. Hence, $A\setminus I'\subseteq F^{-1}gB$ and
$$G=F(A\setminus I')\cup FI'\cup FB\cup I=FF^{-1}gB\cup FB\cup(FI'\cup I),$$
and we conclude that $B$ is $\J$-large.
\end{proof}

\begin{Th}\label{t2.4} If a subset $A$ of $X$ is $\J$-prethick then $\vt_\J(A)$ is large. \end{Th}
\begin{proof} By Lemma~\ref{l2.1}, $A$ is not $\J$-small. 
We take a $\J$-large subset $L$ such that $L\setminus A$ is not $\J$-large. 
Since $L=(L\cap A)\cup(L\setminus A)$, by Lemma~\ref{l2.3}, $\vt_\J(L\cap A)$ is large so $\vt_\J(A)$ is large. \end{proof}

\begin{Cr} If an $\J$-prethick subset $A$ of $X$ is finitely partitioned $A=A_1\cup...A_n$ then $\vt_\J(A_i)$ is large for some $i\in\{1,...,n\}$ \end{Cr}
\begin{proof} By Lemma~\ref{l2.2} some cell $A_i$ is prethick. Apply Theorem~\ref{t2.4}. \end{proof}

\begin{Rm} Given a translation invariant ideal $\J$ in $\PP_X$, there is a function $\Phi_\J:\NN\to\NN$ such that, for any $n$-partition $X_1\cup...\cup X_n$ of $X$, there exists $A_i$ and $F\in[G]^{<\w}$ such that $G=F\vt_\J(A_i)$ and $|F|\le\Phi_\J(n)$. 
These functions are intensively studied in \cite{b2} and \cite{b4}.\end{Rm}

\section{Ultracompanions}

Given a translation invariant ideal $\J$ in $\PP_X$, we denote 
$$\check{\J}=\{p\in X^*: X\setminus I\in p\text{ for each } I\in\J\},$$
and observe that $\check{\J}$ is closed in $X^*$ and $gp\in\check{\J}$ for all $g\in G$ and $p\in\check{\J}$.

\begin{Th} For a subset $A$ of $X$, the following statements hold 
\begin{itemize}
\item[{\it (i)}] $A$ is $\J$-large if and only if $\vm_p(A)\neq\varnothing$ for each $p\in\check{\J}$;
\item[{\it (ii)}] $A$ is $\J$-thick if and only if there exists $p\in\check{\J}$ such that $\vm_p(A)=Gp$;
\item[{\it (iii)}] $A$ is $\J$-prethich if and only if there exists $p\in\check{\J}$ and $F\in[G]^{<\w}$ such that $\vm_p(FA)=Gp$;
\item[{\it (iv)}] $A$ is $\J$-small if and only if for every $p\in\check{\J}$ and every $F\in[G]^{<\w}$, we have $\vm_p(A)\neq Gp$;.
\end{itemize}
\end{Th}
\begin{proof} $(i)$ Suppose that $A$ is $\J$-large and choose $F\in[G]^{<\w}$ and $I\in\J$ such that $X=FA\cup I$. 
We take an arbitrary $p\in\check{\J}$ and choose $g\in F$ such that $gA\in p$ so $A\in g^{-1}p$ and $\vm_p(A)\neq\varnothing$ 

Assume that $\vm_p(A)\neq\varnothing$ for each $p\in\J$. Given $p\in\J$, we choose $g_p\in G$ such that $A\in g_pp$. 
Then we consider a covering of $\check{\J}$ by the subsets $\{g_p^{-1}A^*:p\in\check{\J}\}$ and choose its finite subcovering $g_{p_1}^{-1}A^*,...,g_{p_n}^{-1}A^*$
We take $I\in\J$ and $H\in[X]^{<\w}$ such that $X\setminus (g_{p_1}^{-1}A^*\cup...\cup g_{p_n}^{-1}A^*)=I\cup H$. 
At last, we choose $F\in[G^{<\w}]$ such that $\{g_{p_1}^{-1},...,g_{p_n}^{-1}\}\subseteq F$ and $H\subseteq FA$.
Then $X=FA\cup I$ and $A$ is $\J$-large.

$(ii)$ We note that $A$ is $\J$-thick if and only if $X\setminus A$ is not $\J$-large and apply $(i)$. 

$(iii)$ follows from $(ii)$.

$(iv)$ follows from $(iii)$ and Lemma~\ref{l2.1}.
\end{proof}

We suppose that $\J\neq\{\varnothing\}$ and say that a subset $A$ of $X$ is $\J$-thin if, for every $F\in[G]^{<\w}$, there exists $I\in\J$ such that $|Fa\cap A|\le1$ for each $a\in A\setminus I$.

\begin{Th} A subset $A$ of $X$ is $I$-thin if and only if $\vm_p(A)\le1$ for each $p\in\J$. \end{Th}
\begin{proof} Suppose that $A$ is not $\J$-thin and choose $F\in[G]^{<\w}$ such that, for each $I\in\J$, there is $a_I\in A\setminus I$ satisfying $Fa_I\cap A\neq\{a_I\}$.
We pick $g_I\in F$ and $b_I\in A$ such that $g_Ia_I=b_I$ and $b_I\in A$. 
Then we put $A_I=\{a_{I'}: I\subseteq I', I'\in\J\}$ and take $p\in\check{\J}$ such that $A_I\in p$ for each $I\in\J$.
Since $p$ is an ultrafilter, there exists $g\in F$ such that $gp\neq p$ and $A\in gp$.
Hence $\{p,gp\}\subseteq\vm_p(A)$ and $|\vm_p(A)|>1$.

Assume that $|\vm_p(A)|>1$ for some $p\in\J$.
We pick distinct $g_1p, g_2p\in\vm_p(A)$ and put $F=\{g_2g_1^{-1}\}$. 
Since $A\setminus I\in g_1p\cap g_2p$ for each $I\in\J$, there is $a_I\in A\setminus I$ such that $g_2^{-1}g_1a_I\in A\setminus\{a_I\}$.
Hence, $A$ is not $\J$-thin.
\end{proof}

\begin{Rm} We say that a non-empty subset $S$ of $\beta X^*$ is invariant if $gS\subseteq S$ for each $g\in G$. 
It is easy to see that each closed invariant subset $S$ of $X$ contains a minimal by inclusion closed invariant subset $M$ and $M=cl(Gp)$ for each $p\in M$.
By analogy with Theorem $4.39$ from \cite{b8}, we can prove that for $p\in X^*$ the subset $cl(Gp)$ is minimal if 
and only if, for every $P\in p$, there exists $F\in[G]^{\w}$ such that $Gp\subseteq (FP)^*$.
\end{Rm}

\begin{Rm} Given a translation invariant ideal $\J$ in $\PP_X$, we denote
$$K(\check{\J})=\bigcup\{M:M\text{ is a minimal closed invariant subset of }\check{\J}\}.$$
By analogy with Theorem $4.40$ from \cite{b8}, we can prove that $p\in cl(K(\check{\J}))$ if and only if each subset $P\in p$ is $\J$-prethick.
It is worth to be mentioned that each closed invariant subset $S$ of $X^*$ is of the form $S=\check{J}$ for some translation invariant ideal $\J$ in $\PP_X$.
\end{Rm}

\begin{Rm} By Theorem $6.30$ from \cite{b8}, for every infinite group of cardinality $\varkappa$, there exists $2^{2^\varkappa}$ distinct minimal closed invariant subsets of $G^*$. 
We show that this statement fails to be true for $G$-spaces.
Let $X=\w$ and $G$ be the group of all permutations of $X$.
If $S$ is a closed invariant subset of $X^*$ then $S=X^*$.\end{Rm}

\begin{Rm} We describe a relationship between ultracompanions and relative combinatorial derivations. 
Let $\J$ be a translation invariant ideal in $\PP_X$, $A\subseteq X$, $p\in\check{J}$.
We denote $A_p=\{g\in G: A\in gp\}$ so $\vm_p(A)=A_pp$. Then $$\vt_\J(A)=\bigcap\{A_p^{-1}:p\in\check{\J},A\in p\}.$$\end{Rm}

\section{Isolated points}

Given any $p\in X^*$, we put $$St(p)=\{g\in G: gp=p\},$$
and note that, by \cite[Lemma 3.33]{b8}, $gp=p$ if and only if there exists $P\in p$ such that $gx=x$ for each $x\in P$.

\begin{Th}\label{t4.1} For every $p\in X^*$, the following statements are equivalent 
\begin{itemize}
\item[{\it (i)}] $p$ is not isolated in $Gp$;
\item[{\it (ii)}] there exists $q\in(G\setminus St(p))^*$ such that $qp=p$;
\item[{\it (iii)}] there exists $\varepsilon\in(G\setminus St(p))^*$ such that $\varepsilon\varepsilon=\varepsilon$ and $\varepsilon p=p$.
\end{itemize}\end{Th}
\begin{proof} 
The implications $(i)\Rightarrow(ii)$ and $(iii)\Rightarrow(i)$ are evident.

$(ii)\Rightarrow(iii)$. In view of Theorem $2.5$ from \cite{b8}, it suffices to show that the set
$$S=\{q\in(G\setminus St(p))^*:qp=p\}$$ is a subsemigroup of $G^*$.
Let $q,r\in S$, $Q\in q$. For each $x\in Q$, we choose $R_x\in r$ such that $x^{-1}St(p)\cap R_x=\varnothing$.
Then $xy\notin St(p)$ for each $y\in R_x$. 
We put $$P=\bigcup_{x\in Q}xR_x,$$ and note  that $P\in qr$ and $P\cap St(p)=\varnothing$. Hence $qr\in S$.
\end{proof}

\begin{Rm} For each $g\in G$, the mapping $p\mapsto gp:\text{ }\beta X\to\beta X$ is a homeomorphism. 
It follows that $Gp$ has an isolated point if and only if $Gp$ is discrete.\end{Rm}

Let $(g_n)_{n\in\w}$ be sequence in $G$ and let $(x_n){n\in\w}$ be a sequence in $X$ such that
\begin{itemize}
\item[(1)] $\{g_0^{\e_0}...g_n^{\e_n}x_n:\e_i\in\{0,1\}\}\cap\{g_0^{\e_0}...g_n^{\e_m}x_m:\e_i\in\{0,1\}\}=\varnothing$ for all distinct $m,n\in\w$;
\item[(2)] $|\{g_0^{\e_0}...g_n^{\e_n}x_n:\e_i\in\{0,1\}\}|=2^{n+1}$ for every $n\in\w$.
\end{itemize}

We say that a subset $Y$ of $X$ is a {\em piecewise shifted $FP$-set} if there exist 
$(g_n)_{n\in\w}$, $(x_n)_{n\in\w}$ satisfying $(1)$ and $(2)$ such that
$$Y=\{g_0^{\e_0}...g_n^{\e_n}x_n:\e_i\in\{0,1\}, n\in\w\}.$$
For definition of an $FP$-set in a group see \cite[p. 108]{b8}.

\begin{Th}\label{t4.3} Let $p$ be an ultrafilter from $X^*$ such that $Gp$ is not discrete. Then every subset $P\in p$ contains a piecewise shifted $FP$-set. \end{Th}
\begin{proof} We choose $g_0\in G$ such that $p\neq g_0p$, $P\in g_0p$ and take $P_0\subseteq P$, $P_0\in p$ such that $g_0P_0\cap P_0=\varnothing$.
We pick an arbitrary $x_0\in P_0$.

Suppose that the elements $g_0,...,g_n$ from $G$ and $x_0,...,x_n$ from $X$ have been chosen so that
\begin{itemize}
\item[(3)] $g_0^{\e_0}...g_k^{\e_k}x_k\in P$ for all $\e_i\in\{0,1\}$ and $k\le n$;
\item[(4)] $\{g_0^{\e_0}...g_k^{\e_k}x_k:\e_i\in{0,1}\}\cap\{g_0^{\e_0}...g_m^{\e_m}x_m:\e_i\in\{0,1\}\}=\varnothing$ for all $k<m\le n$;
\item[(5)] $|\{g_0^{\e_0}...g_k^{\e_k}x_k:\e_i\in{0,1}\}|=2^{k+1}$ for all $k\le n$;
\item[(6)] $P\in g_0^{\e_0}...g_k^{\e_k}p$ for all $\e_i\in\{0,1\}$ and $k\le n$;
\item[(7)] $|\{g_0^{\e_0}...g_k^{\e_k}p:\e_i\in{0,1}\}|=2^{k+1}$ for all $k\le n$.
\end{itemize}

Since $p$ is not isolated in $Gp$, we use $(6)$ and $(7)$ to choose $g_{n+1}\in G$ such that $P\in g_0^{\e_0}...g_{n+1}^{\e_{n+1}}p$ for all $\e_i\in\{0,1\}$ and $|\{g_0^{\e_0}...g_{n+1}^{\e_{n+1}}p:\e_i\in\{0,1\}\}|=2^{n+2}$.

Then we choose $P_{n+1}\in p$ such that $g_0^{\e_0}...g_{n+1}^{\e_{n+1}}P_{n+1}\subseteq P$ for all $\e_i\in\{0,1\}$ and
$$g_0^{\e_0}...g_{n+1}^{\e_{n+1}}P_{n+1}\cap g_0^{\delta_0}...g_{n+1}^{\delta_{n+1}}P_{n+1}=\varnothing$$
for all distinct $(\e_0,...,\e_{n+1})$ and $(\delta_0,...,\delta_{n+1})$ from $\{0,1\}^{n+2}$

We pick $x_{n+1}\in P_{n+1}$ so that 
$$\{g_0^{\e_0}...g_{n+1}^{\e_{n+1}}x_{n+1}:\e_i\in\{0,1\}\}\cap\{g_0^{\e_0}...g_k^{\e_k}x_k:\e_i\in\{0,1\}\}=\varnothing$$ for each $k\le n$.

After $\w$ steps, we get the sequences $(g_n)_{n\in\w}$ and $(x_n)_{n\in\w}$ which define the desired $FP$-set in $P$.
\end{proof}

\begin{Th}\label{t4.4} For an infinite subset $A$ of a $G$-space $X$, the following statements are equivalent
\begin{itemize}
\item[{\it (i)}] $Gp$ is discrete for each $p\in A^*$;
\item[{\it (ii)}] $A$ contains no  piecewise shifted $FP$-sets.
\end{itemize}\end{Th}
\begin{proof}
The implication $(ii)\Rightarrow(i)$ follows from Theorem~\ref{t4.3}.
To prove $(i)\Rightarrow(ii)$, we suppose that $A$ contains a piecewise shifted $FP$-set $Y$ defined by the sequence $(g_n)_{n\in\w}$ and $(x_n)_{n\in\w}$.
By \cite[Theorem 5.12]{b8}, there is an idempotent $\e\in G^*$ such that, for each $m\in\w$,
$$\{g_m^{\e_m}...g_n^{\e_n}:\e_i\in\{0,1\},m<n<\w\}\in\e.$$

We take an arbitrary $q\in A^*$ such that $\{x_n:n\in\w\}\in q$. Put $p=\e q$.
Since $Y\subseteq A$, we have $p\in A^*$. Clearly, $\e p=p$. 
We note that $g_m^{\e_m}...g_n^{\e_n}\in St(p)$ if and only if $\e_m=...=\e_n=0$.
Hence $G\setminus St(p)\in\e$ and, applying Theorem~\ref{t4.1}, we conclude that $p$ is not isolated in $Gp$.
\end{proof}

\section{Scattered and sparse subsets of $G$-spaces}

Given $F\in[G]^{<\w}$ and $x\in X$, we denote $B(x,F)=Fx\cup\{x\}$ and say that $B(x,F)$ is a {\em ball of radius $F$ around $x$}.
For subset $Y$ of $X$ and $y\in Y$, we denote $B_Y(y,F)=B(y,F)\cap Y$.

A subset $A$ of $X$ is called
\begin{itemize}
\item{} {\em scattered} if, for every infinite subset $Y$ of $X$, there exists $H\in[G]^{<\w}$ such that, for every $F\in[G]^{<\w}$ there is $y\in Y$ such that $B_Y(y,F)\cap B_Y(y,H)=\varnothing$;
\item{} {\em sparse} if, for every infinite subset $Y$ of $X$, there exists $H\in[G]^{<\w}$ such that, for every $F\in[G]^{<\w}$ there is $y\in Y$ such that $B_A(y,F)\cap B_A(y,H)=\varnothing$.
\end{itemize}

Clearly, each sparse subset is scattered. The sparse subsets of groups were introduced in \cite{b7} and studied in \cite{b9} \cite{b10}.
From the asymptotic point of view \cite{b16}, the scattered subsets of $G$-spaces can be considered as counterparts of the scattered subspaces of topological spaces.

\begin{Ps}  A subset $A$ of a $G$-space $X$ is sparse if and only if $\vm_p(A)$ is finite for each $p\in X^*$.\end{Ps}
\begin{proof} Repeat the arguments proving Theorem $10$ in \cite{b14}.
\end{proof}

\begin{Ps} A subset $A$ of a $G$-space $X$ is scattered if and only if, for every infinite subset $Y$ of $X$, there exists $p\in Y^*$ such that $\vm_p(Y)$ is finite.\end{Ps}
\begin{proof} Repeat the arguments proving Proposition $1$ in \cite{b3}.
\end{proof}

To formulate further results, we need some asymptology (see \cite[Chapter 1]{b16}).
Let $G_1$, $G_2$ be groups, $X_1$ be a $G_1$-space, $X_2$ be a $G_2$-space, $Y_1\subseteq X_1$, $Y_2\subseteq X_2$.
A mapping $f:Y_1\to Y_2$ is called a {\em $\prec$-mapping} if, for every $F\in[G_1]^{<\w}$, there exists $H\in[G_2]^{<\w}$ such that, for every $y\in Y_1$
$$f(B_{Y_1}(y,F))\subseteq B_{Y_2}(f(y),H).$$
If $f$ is a bijection such that $f$ and $f^{-1}$ are $\prec$-mappings, we say that $f$ is an {\em asymorphism}. 
The subset subsets $Y_1$ and $Y_2$ are {\em coarsely equivalent} if there exist asymorphic subsets $Z_1\subseteq Y_1$, $Z_2\subseteq Y_2$ such that $Y_1=B_{Y_1}(Z_1,F)$, $Y_2=B_{Y_2}(Z_2,H)$ for some $F\in[G_1]^{<\w}$, $H\in[G_2]^{<\w}$.
We say that a property $\PP$ of subsets of $G$-spaces is {\em coarse} if $\PP$ is stable under coarse equivalent, and note that "sparse" and "scattered" are coarse properties.

In asymptology, the group $\oplus_\w\ZZ_2$ is known under name the Cantor macrocube, for its coarse characterization see \cite{b5}.

\begin{Th} A subset $A$ of a $G$-space $X$ is sparse if and only if $A$ has no subsets asymorphic to the subset $W_2=\{g\in\oplus_\w\ZZ_2: supt g\le2\}$ of the Cantor macrocube.\end{Th}
\begin{proof} Apply arguments from \cite[Proof of Theorem 3]{b14}. \end{proof}

\begin{Th}\label{t5.4} For a subset $A$ of a $G$-space $X$, the following statements are equivalent 
\begin{itemize}
\item[{\it (i)}] $A$ is scattered;
\item[{\it (ii)}] $\vm_p(A)$ is discrete for each $p\in X^*$;
\item[{\it (iii)}] $A$ contains no piecewise shifted $FP$-sets;
\item[{\it (iv)}] $A$ contains no subsets coarsely equivalent to the Cantor macrocube.
\end{itemize}\end{Th}
\begin{proof} The equivalence $(ii)\Rightarrow(iii)$ follows from Theorem~\ref{t4.4}. 
To prove $(i)\Rightarrow(iii)$, repeat the arguments from \cite[Proof of Theorem 1]{b3}.

$(ii)\Rightarrow(i)$. Let $Y$ be an infinite subset of $A$. 
We denote by $\FF$ the family of all closed invariant subsets of $X^*$ and put $\FF_Y=\{F\cap Y^*:F\in\FF\}$.
By the Zorn's lemma, there exists minimal by inclusion element $M\in\FF_Y$.
We take an arbitrary $p\in M$ and show that $\vm_p(Y)$ is finite. Assume the contrary.
Then the set $\vm_p(Y)$ has a limit point $q$. Since $M$ is minimal and $p\in M$, there exists $r\in\beta G$ such that $p=rq$.
By the definition of the action of $\beta G$ on $\beta X$, for every $P\in p$, there exists $Q\in q$ and $g\in G$ such that $gQ\subseteq P$.
It follows that $p$ is a limit point of $\vm_p(Y)$. Hence, $\vm_p(Y)$ is not discrete and we get a contradiction.

The implication $(i)\Rightarrow(iv)$ is evident because the Cantor macrocube is not scattered. 
To prove $(iv)\Rightarrow(i)$, we use the characterization of the Cantor macrocube from \cite{b5} and the arguments from \cite[Proof of the Proposition 3]{b3}.
\end{proof}

\begin{Rm} Let $G$ be an amenable group, $A$ be scattered subset of $G$. 
By \cite[Theorem 2]{b3}, $\mu(A)=0$ for each left invariant Banach measure $\mu$ on $G$.
This statement cannot be extended to all $G$-spaces. 
As a counterexample, we take $X=\w$ and $G$ is a group of all permutations of $X$ with finite supports.
In this case, each subset of $X$ is scattered. \end{Rm}

Let $X$ be a $G$-space, $\J$ be a translation invariant ideal in $\PP_X$. We say that a subset $A$ of $X$ is
\begin{itemize}
\item{} {\em $\J$-sparse} if $\vm_p(A)$ is finite for each $p\in\check{\J}$;
\item{} {\em $\J$-scattered} if, for every subset $Y$ of $A$, $Y\notin\check{\J}$, there is $p\in\check{\J}\cap Y^*$ such that $\vm_p(Y)$ is finite.
\end{itemize}
In this context, sparse and scattered subsets coincide with $[X]^{<\w}$-sparse and $[X]^{<\w}$-scattered subsets respectively.

The arguments proving $(ii)\Rightarrow(i)$ in Theorem~\ref{t5.4} witness that $A$ is scattered provided that each point $p\in\check{\J}\cap A^*$ is isolated in $X^*$.

\begin{Qs} Assume that $A$ is $J$-scattered. Is every point $p\in\check{\J}\cap A^*$ isolated in $X^*$? \end{Qs}

If a subset $A$ of $X$ has a subset $Y\notin\J$ coarsely equivalent to $\oplus_\w\ZZ_2$ then $A$ is not $\J$-scattered.

\begin{Qs} Assume that a subset $A$ of $X$ has no subsets $Y\notin\J$ coarsely equivalent to $\oplus_\w\ZZ_2$. Is $A$ $J$-scattered? \end{Qs}

We note that the families $\sigma(\J)$ and $\partial(\J)$ of all $\J$-sparse and $\J$-scattered subsets of $X$ are translation invariant ideals in $\PP_X$ and say that $\J$ is {\em $\sigma$-complete} (resp. {\em $\partial$-complete}) if $\sigma{\J}=\J$ (resp $\partial(\J)=\J$).
We denote by $\sigma^*(\J)$ (resp. $\partial^*(\J)$) the intersection of all $\sigma$-complete (resp $\partial$-complete) ideals containing $\J$.
Clearly, $\sigma^*(\J)$ and $\partial^*(\J)$ are the smallest $\sigma$-complete and $\partial$-complete ideals such that $\J\subseteq\sigma^*(\J)$ and $\J\subseteq\partial^*(\J)$. 
We say that $\sigma^*(\J)$ and $\partial^*(\J)$ are the {\em $\sigma$-completion} and {\em $\partial$-completion} of $\J$ respectively.

We define a sequence $(\sigma^n(\J))_{n<\w}$ by the recursion: $\sigma^0(\J)=\J$, $\sigma^{n+1}(\J)=\sigma(\sigma^n(\J))$, and note that $\bigcup_{n\in\w}\sigma^n(\J)\subseteq\sigma^*(\J)$.
If $X$ is left regular, by \cite[Theorem 4(1)]{b10}, $\sigma^*(\J)=\bigcup_{n\in\w}\sigma^n(\J)$ and by \cite[Theorem 4(2)]{b10}, $\sigma^{n+1}([G]^{<\w})\neq\sigma^n([G]^{<\w})$ for each $n\in\w$.

\begin{Qs} Is $\sigma^*\J)=\bigcup_{n\in\w}\sigma^n(\J)$ for each translation invariant ideal $\J$ in an arbitrary $G$-space $X$? \end{Qs}

In contrast to $\sigma$-completion, for each translation invariant ideal $\J$ in $\PP_X$, we have $\partial^*(\J)=\partial(\J)$.
In particular the ideal $\partial([X]^{<\w})$ of all sparse subsets of $X$ is $\partial$-complete.
Indeed, assume that $A\notin\partial(\J)$ and choose $Y\subseteq A$, $Y\notin\J$ such that $\vm_p(Y)$ is infinite for each $p\in\check{\J}\cap Y^*$.
Then $Y\notin\partial(Y)$ and $A\notin\partial^2(\J)$.
Hence, $\partial^2(\J)=\partial(\J)$ so $\partial^*(\J)=\partial(\J)$.


\begin{thebibliography}{99}
\bibitem{b1} T.~Banakh, N.~Lyaskovska, {\it Completeness of translation-invariant ideals in groups}, Ukr. Math. J. {\bf 62} (2010), 1022-1031.
\bibitem{b2} T.~Banakh, I.~Protasov, S.~Slobodianiuk, {\it Densities, submeasures and partitions of groups}, preprint (http://arxiv.org/abs/1303.4612).
\bibitem{b3} T.~Banakh, I.~Protasov, S.~Slobodianiuk, {\it Scattered subsets of groups} preprint (http://arxiv.org/abs/1312.6946).
\bibitem{b4} T.~Banakh, O.~Ravsky, S.~Slobodianiuk, {\it On partitions of $G$-spaces and $G$-lattices}, preprint (http://arxiv.org/abs/1303.1427).
\bibitem{b5} T.~Banakh, I.~Zarichnyi, {\it Characterizing the Cantor bi-cube in asymptotic categories,} Groups, Geometry and Dynamics {\bf 5} (2011), 691-728.
\bibitem{b6} J.~Erde, {\it A note on combinatorial derivation}, preprint (http://arxiv.org/abs/1210.7622).
\bibitem{b7} M. Filali, Ie.~Lutsenko, I.~Protasov, {\it Boolean group ideals and the ideal structure of $\beta G$}, Math. Stud. {\bf 30} (2008) 1-10.
\bibitem{b8} N.~Hindman, D.~Strauss {\it Algebra in the Stone-$\check{C}$ech compactification}, 2nd edition, de Gruyter, 2012.
\bibitem{b9} Ie.~Lutsenko, I.V.~Protasov, {\it Sparse, thin and other subsets of groups}, Intern. J. Algebra Computation {\bf 19} (2009) 491-510.
\bibitem{b10} Ie.~Lutsenko, I.V.~Protasov, {\it Relatively thin and sparse subsets of groups}, Ukr. Math. J. {\bf 63} (2011), 216-225.
\bibitem{b11} I.V.~Protasov, {\it Selective survey on Subset Combinatorics of Groups}, Ukr. Math. Bull. {\bf 7} (2011), 220-257.
\bibitem{b12} I.V.~Protasov, {\it The Combinatorial Derivation}, Appl. Gen. Topology {\bf 14}, 2 (2013), 171-178.
\bibitem{b13} I.V.~Protasov, {\it The combinatorial derivation and its inverse mapping}, Central Europ. Math. J. {\bf 11} (2013), 2176-2181.
\bibitem{b14} I.V.~Protasov, {\it Sparse and thin metric spaces}, Math. Stud. (to appear).
\bibitem{b15} I.V.~Protasov, S.~Slobodianiuk, {\it Ultracompanions of subsets of groups}, Comment. Math. Univ. Carolin. (to appear), preprint (http://arxiv.org/abs/1308.1497).
\bibitem{b16} I.~Protasov, M.~Zarichnyi, {\it General Asymptology}, Math. Stud. Monogr. Ser., Vol. 12, VNTL, Lviv, 2007.

\end{thebibliography}
\end{document}